\documentclass[10pt]{amsart}
\usepackage{amscd}
\usepackage{amsfonts}
\usepackage{amsmath}
\usepackage{amssymb}
\usepackage{bbm}
\usepackage{blkarray}
\usepackage{easybmat}
\usepackage{enumerate}
\usepackage{extarrows}
\usepackage{fancybox}
\usepackage{graphicx}
\usepackage{indentfirst}
\usepackage{mathbbol}
\usepackage{mathrsfs}
\usepackage{mathtools}
\usepackage{paralist}
\usepackage{picinpar}
\usepackage{url}
\usepackage{varioref}
\usepackage{wrapfig}
\usepackage{xcolor}
\usepackage[colorlinks, linkcolor=black, anchorcolor=black, citecolor=black, urlcolor=black]{hyperref}

\newtheorem{thm}{Theorem}[section]
\newtheorem{prop}[thm]{Proposition}
\newtheorem{lem}[thm]{Lemma}
\newtheorem{cor}[thm]{Corollary}
\theoremstyle{definition}

\theoremstyle{remark}
\newtheorem*{rmk}{Remark}

\theoremstyle{definition}
\newtheorem{defn}[thm]{Definition}

\numberwithin{equation}{section}

\begin{document}

\title[Restriction formula]{On the restriction formula}

\thanks{}

\author[Xiankui Meng]{Xiankui Meng}
\address{Xiankui Meng: School of Science, Beijing University of Posts and Telecommunications, Beijing 100876, People's Republic of China.}
\email{mengxiankui@amss.ac.cn}

\author[Xiangyu Zhou]{Xiangyu Zhou}
\address{Xiangyu Zhou: Institute of Mathematics, AMSS, and Hua Loo-Keng Key Laboratory of Mathematics, Chinese Academy of Sciences, Beijing, 100190, China}
\email{xyzhou@math.ac.cn}

\subjclass[2010]{}


\dedicatory{}

\date{}

\begin{abstract}
Let $\varphi$ be a quasi-psh function on a complex manifold $X$ and let $S\subset X$ be a complex submanifold.
Then the multiplier ideal sheaves $\mathcal{I}(\varphi|_S)\subset\mathcal{I}(\varphi)|_{S}$ and the complex singularity exponents $c_{x}\left(\varphi|_{S}\right)\leqslant c_{x}(\varphi)$ by Ohsawa-Takegoshi $L^{2}$ extension theorem.
An interesting question is to know whether it is possible to get equalities in the above formulas.
In the present article, we show that the answer is positive when $S$ is chosen outside a measure zero set in a suitable projective space.
\end{abstract}

\maketitle

%%%%%%%%%%%%%%%%%%%%%%%%%%%%%%%%%%%%%%%%%%%%%%%%%%%%%%%%%%%%

\section{ Introduction}

Let $X$ be a complex manifold.
Throughout this paper, our complex manifolds are always assumed to be second countable.
A quasi-plurisubharmonic (quasi-psh) function on $X$ is by definition a function $\varphi$ which is locally equal to the sum of a plurisubharmonic function and of a smooth function.
The multiplier ideal sheaf $\mathcal{I}(\varphi)$
is the ideal subsheaf of $\mathscr{O}_{X}$ defined by
\[\mathcal{I}(\varphi)_{x}=\{f\in\mathscr{O}_{X, x}; ~|f|^{2}e^{-2\varphi} ~\text{is integrable in a neighborhood of}~ x\},\]
It's well-known that $\mathcal{I}(\varphi)$ is a coherent analytic sheaf.

Another related invariant is complex singularity exponent.
For any compact set $K\subset X$, we introduce the complex singularity exponent of $\varphi$ on $K$ to be the nonnegative number
\begin{equation}
c_{K}(\varphi)=\sup\left\{c\geqslant 0; \quad e^{-2c\varphi} ~\text{is}~L^{1} ~\text{on a neighborhood of}~ K\right\}.
\end{equation}
If $\varphi=-\infty$ near some connected component of $K$, we put of course $c_{K}(\varphi)=0$.
Given a point $x\in X$, we write $c_{x}(\varphi)$ instead of $c_{\{x\}}(\varphi)$.

Multiplier ideal sheaves and complex singularity exponents associated to quasi-psh functions are basic objects in several complex variables and algebraic geometry. 
The restriction formulas for multiplier ideal sheaves and for complex singularity exponents are useful in the inductive arguments.
Various important and fundamental properties about them have been established by \cite{DEL00}, \cite{DeKo01}, \cite{Cao14}, \cite{FuMa16}, \cite{Fuj17}, \cite{Guan-Zhou20}, etc.
In the present article, we will discuss the restriction formulas in a more general setting.

We first give a generalized version of the Bertini theorem.
\begin{thm}\label{thm: MA}
Let $F$ be a holomorphic vector bundle over a complex manifold $X$ and let $W$ be a finite dimensional subspace of $H^{0}\left(X, F\right)$ such that $W$ generates all fibers $F_{x}, ~x\in X$.
Let $P(W)$ denote the projective space of $W$.
If $W$ has no non-vanishing sections, then $\dim X\geqslant\operatorname{rank}F$ and the set
\[Z=\Big\{[s]\in P(W) \big| ~s^{-1}(0) ~\text{is not smooth}\Big\}\]
has Lebesgue measure zero in $P(W)$. 
If, moreover, $X$ is compact, then $Z$ is analytic in $P(W)$.
\end{thm}

The restriction formula on multiplier ideal sheaves is also a Bertini type result.
Motived by \cite{Cao14}, we shall give a simple approach to this problem. 
The main idea of the proof is to apply the Fubini theorem.

\begin{thm}\label{thm: MB}
Let $\varphi$ be a quasi-psh function on a complex manifold $X$.
Let $F$ be a holomorphic vector bundle over $X$ and let $W$ be a finite dimensional subspace of $H^{0}\left(X, F\right)$ such that $W$ generates all fibers $F_{x}, ~x\in X$.
Let $P(W)$ denote the projective space of $W$ and 
\begin{equation*}
B=\Big\{[s]\in P(W) \big| ~S=s^{-1}(0) ~\text{is smooth and}~\mathcal{I}(\varphi|_S)=\mathcal{I}(\varphi)|_{S}\Big\}.
\end{equation*}
If $W$ has no non-vanishing sections, then $P(W)\setminus B$ has measure zero in $P(W)$.
\end{thm}

A special case of Theorem \ref{thm: MB} is the following statement.
\begin{cor}\label{cor: thin}
Let $\mathfrak{d}$ be a base point free linear system on a complex manifold $X$ with $1\leqslant\dim\mathfrak{d}<+\infty$, and let $\phi$ be a quasi-plurisubharmonic function on $X$. 
If we put
\[\mathfrak{b}=\Big\{S\in\mathfrak{d}; \quad S~\text{is smooth and} ~\mathcal{I}(\varphi)|_{S}=\mathcal{I}(\varphi|_{S}) \Big\},\]
then $\mathfrak{d}\setminus\mathfrak{b}$ has measure zero in $\mathfrak{d}$.
\end{cor}

In case $X$ is compact, Fujino and Matsumura proved that $\mathfrak{b}$ is dense in $\mathfrak{d}$ in the classical topology (cf. Theorem 1.10 in \cite{FuMa16}).
However, their arguments give no information on the set $\mathfrak{d}\setminus\mathfrak{b}$.
The main idea of our approach is to use Fubini's theorem and hence we can prove the set $\mathfrak{d}\setminus\mathfrak{b}$ has measure zero.
In fact, there is a problem due to S\'ebastien Boucksom:
is $\mathfrak{d}\setminus\mathfrak{b}$ a pluripolar set?

The restriction formula on complex singularity exponents is given by the following result.
\begin{thm}
Let $\varphi$ be a quasi-psh function on a complex manifold $X$.
Let $F$ be a holomorphic vector bundle over $X$ and let $W$ be a finite dimensional subspace of $H^{0}\left(X, F\right)$ such that $W$ generates all fibers $F_{x}, ~x\in X$.
Let $P(W)$ denote the projective space of $W$ and 
\begin{equation*}
Q=\Big\{[s]\in P(W) \big| ~S=s^{-1}(0) ~\text{is smooth and}~c_{x}\left(\varphi|_S\right)=c_{x}(\varphi), ~\forall x\in S\Big\}.
\end{equation*}
If $W$ has no non-vanishing sections, then $P(W)\setminus Q$ has measure zero in $P(W)$.
\end{thm}

The restriction formula on multiplier ideal sheaves can be used to deduced an important exact sequence which is useful in the inductive arguments.

\begin{thm}
Let $L$ be a holomorphic line bundle over a complex manifold $X$ 
and let $W$ be a finite dimensional subspace of $H^{0}\left(X, L\right)$ such that $W$ generates all fibers $L_{x}, ~x\in X$. 
We denote by $P(W)$ the projective space of $W$.
Let $\varphi$ be a quasi-psh function on $X$ and let $Y_{i}, ~i\in I$ be the analytic subsets associated to $\mathscr{O}_{X}/\mathcal{I}(\varphi)$.
Suppose $L$ is not a trivial line bundle. Then:
\begin{itemize}
\item
The set
\[A=\Big\{[s]\in P(W) \big| ~s^{-1}(0)\supset Y_{i} ~\text{for some}~i\in I\Big\}\]
is a countable union of proper analytic subsets of $P(W)$.
If, moreover, $X$ is compact, then $A$ is analytic in $P(W)$.
\item
Suppose $[s]\in P(W)$ and $S=s^{-1}(0)$, then the sequence
\begin{equation*}
 0\longrightarrow\mathcal{I}(\varphi)\otimes\mathscr{O}(-S)\longrightarrow\mathcal{I}(\varphi)\longrightarrow\mathcal{I}(\varphi)|_{S}\longrightarrow 0
\end{equation*}
is exact if and only if $[s]\notin A$.
\end{itemize}
\end{thm}

\begin{cor}
Let $\mathfrak{d}$ be a base point free linear system on a complex manifold $X$ with $1\leqslant\dim\mathfrak{d}<+\infty$, and let $\phi$ be a quasi-plurisubharmonic function on $X$. 
Then there exists a measure zero set $\mathfrak{n}\subset \mathfrak{d}$ such that for each $S\in\mathfrak{d}\setminus\mathfrak{n}$ the divisor $S$ is smooth, $\mathcal{I}(\varphi|_S)=\mathcal{I}(\varphi)|_{S}$ and the sequence
\begin{equation}
 0\longrightarrow\mathcal{I}(\varphi)\otimes\mathscr{O}(-S)\longrightarrow\mathcal{I}(\varphi)\longrightarrow\mathcal{I}(\varphi|_{S})\longrightarrow 0
\end{equation}
is exact.
\end{cor}

The above corollary can also be obtained by the approach of Cao (\cite{Cao14}).
Following the arguments of Cao, we can show the following result.

\begin{thm}
Let $\varphi$ be a quasi-psh function on a complex manifold $X$.
Suppose $\sigma_{0}$ is a positive continuous function on $X$
such that $\mathcal{I}(\varphi)=\mathcal{I}\left((1+\sigma_{0})\varphi\right)$.
Let $L$ be a holomorphic line bundle over $X$. 
Let $W$ be a finite dimensional subspace of $H^{0}\left(X, L\right)$ such that $W$ generates all fibers $L_{x}, ~x\in X$. 
Then there exists a measure zero set $N\subset P(W)$ such that for each $[s]\in P(W)\setminus Z$ the multiplier ideal sheaf $\mathcal{I}(\varphi)$ can be written as
\begin{equation}\label{eq:adj}
\mathcal{I}(\varphi)_{x}=\left\{f\in\mathscr{O}_{X, x}; ~\exists ~U_{x} ~\text{such that} ~\int_{U_{x}}\frac{|f|^{2}}{|s|^{2(1-\varepsilon)}}e^{-2(1+\sigma)\varphi}\mathrm{d}V<+\infty\right\},
\end{equation}
for $0<\sigma\leqslant\sigma_{0}$ and $0<\varepsilon<\sigma$, where $dV$ is a smooth volume form on $X$,
\end{thm}

The existence of $\sigma_{0}$ is guaranteed by the solution of the strong openness conjecture (cf. \cite{Guan-Zhou15a}, \cite{Guan-Zhou15b}).

%%%%%%%%%%%%%%%%%%%%%%%%%%%%%%%%%%%%%%%%%%%%%%%%%%%%%%%%

\section{Bertini type theorem}

Let us recall a well-known fact.
\begin{thm}[cf. \cite{{GPR94}}]\label{thm: Sard}
Let $F: X\to Y$ be a surjective holomorphic map between complex manifolds 
and let $X_{y}=F^{-1}(y)$ be the "full" fiber over a point $y\in Y$
(i.e. $F^{-1}(y)$ is equipped with the structure sheaf coming from $\operatorname{Im}\left(F^{\ast}(\mathbb{m}_{y})\to\mathscr{O}_{X}\right)$).
Then the set \[Z=\{y\in Y; ~ X_{y} ~\text{is not smooth}\}\] has Lebesgue measure zero in $Y$, 
and the tangent map $F_{\ast}: T_{X}\to T_{Y}$ is surjective at each point of $X\setminus F^{-1}(Z)$.
If, moreover, $F$ is proper, then $Z$ is analytic in $Y$.
\end{thm}

\begin{thm}\label{thm: Ber}
Let $F$ be a holomorphic vector bundle over a complex manifold $X$ and let $W$ be a finite dimensional subspace of $H^{0}\left(X, F\right)$ such that $W$ generates all fibers $F_{x}, ~x\in X$.
Let $P(W)$ denote the projective space of $W$.
If $W$ has no non-vanishing sections, then $\dim X\geqslant\operatorname{rank}F$ and the set
\[Z=\Big\{[s]\in P(W) \big| ~s^{-1}(0) ~\text{is not smooth}\Big\}\]
has Lebesgue measure zero in $P(W)$. 
If, moreover, $X$ is compact, then $Z$ is analytic in $P(W)$.
\end{thm}

\begin{proof}
We denote by $r$ the rank of $F$ and by $N$ the dimension of $W$.
We first note that $\dim W\geqslant r$ since $F$ is generated by sections in $W$.
However, if $\dim W=r$, then one can find a basis $s_{1}, \cdots, s_{r}\in W$ such that $F$ is generated by these $r$ sections. 
In this case $s_{1}(x), \cdots, s_{r}(x)\in F_{x}$ are linearly independent at each point $x\in X$ and hence $s_{1}, \cdots, s_{r}$ are non-vanishing sections in $W$. 
Therefore, we may assume that $N\geqslant r+1$.

Let $\underline{W}=X\times W$ be the trivial vector bundle of rank $N$ over $X$.
Since $F$ is generated by sections in $W$, we have a surjective bundle morphism
\begin{equation*}
\underline{W}\xlongrightarrow{\Phi} F\longrightarrow 0.
\end{equation*}
Here, the bundle morphism $\Phi$ is given by the evaluation map
\begin{equation*}
\left(x, s\right)\mapsto s(x), \quad x\in X, ~s\in W.
\end{equation*}
Then we have the following exact sequence of vector bundles
\[0\to E\to \underline{W}\to F\to 0,\]
where $E$ is a holomorphic subbundle of $\underline{W}$ and the fibers of $E$ are 
\[E_{x}=\left\{s\in W ; ~ s(x)=0 \right\}, ~x\in X.\]
Let us consider the projectivized bundles \[P\left(\underline{W}\right)=X\times P(W).\] 
The points of $P\left(\underline{W}\right)$ can be identified with the lines in the fibers of $\underline{W}$.
The elements of $P\left(\underline{W}\right)$ can be written as
\[\left(x, [s]\right)\in P\left(\underline{W}\right), \quad x\in X, ~s\in W.\]
If we define
\[P(E)=\coprod_{x\in X}P(E_{x}),\]
then \[\pi: P(E)\longrightarrow X\] is the projectivized bundle of $E$. 
Since $E$ is a subbundle of $\underline{W}$, we have
\[P(E)\subset P\left(\underline{W}\right)\]
and the points of $P(E)$ can be represented by
\[\Big\{\left(x, [s]\right)\in P\left(\underline{W}\right); \quad s\in W, ~s(x)=0.\Big\}\]
The natural projection 
\[\mu: P\left(\underline{W}\right)=X\times P(W)\to P(W), \quad \left(x, [s]\right)\mapsto [s]\]
induces a holomorphic map
\[\mu: P(E)\longrightarrow P(W).\]
Thus we obtain the following diagram
\begin{equation}
\begin{CD}
P(E) @>>{\mu}> P(W) \\
@VV\pi V \\
X 
\end{CD}
\end{equation}

We claim that $\mu$ is surjective.
In fact, for any $s\in W$, we can find a point $x_{0}\in X$ such that $s(x_{0})=0$ 
since $W$ has no non-vanishing sections.
Therefore, for any point $[s]\in P(W)$, we can find a point $\left(x_{0}, [s]\right)\in P(E)$ such that $\mu\left(x_{0}, [s]\right)=[s]$.
The surjectivity of $\mu$ yields $\dim P(E)\geqslant P(W)$.
The dimensions of $P(E)$ and $P(W)$ are $\dim X+N-r-1$ and $N-1$ respectively.
So we can conclude that $\dim X\geqslant r$.

We next consider the fibers of $\mu$.
If $[s]\in P(W)$, then the fiber
\[\mu^{-1}\left([s]\right)=\left\{\left(x, [s]\right)\in X\times P(W);~s(x)=0\right\},\]
and it is isomorphic to a subvariety $S=\left\{x\in X; ~s(x)=0\right\}$ of $X$.
Here, the isomorphism is given by the projection $\pi: P(E)\to X$.
In what follows we do not distinguish the fiber $\mu^{-1}\left([s]\right)$ and the subvariety $S$.

If $X$ is compact, then the map $\mu: P(E)\to P(W)$ is proper.
By Theorem \ref{thm: Sard}, one can find a proper analytic subset $Z\subset P(W)$ such that
\[\mu: P(E)\setminus\mu^{-1}(Z)\longrightarrow P(W)\setminus Z\] 
is a submersion and the fibers $\mu^{-1}\left([s]\right)\cong s^{-1}(0)$, ~$[s]\in P(W)\setminus Z$ are smooth.
The proper mapping theorem implies $Z$ is an analytic subset of $P(W)$ in case $\mu$ is proper.

In the general case, the holomorphic map $\mu: P(E)\to P(W)$ may not be proper.
Let $C$ denote the critical set of $\mu$, which is the set of points in $X$ at which the differential $\mu_{\ast}: T_{P(E)}\to T_{P(W)}$ of $\mu$ is not surjective.
It is obvious that $C$ is an analytic subset of $P(E)$.
By Sard's theorem, the set $\mu(C)$ has measure zero in $P(W)$.
Moreover, one can show that $\mu(C)$ is a countable union of nowhere dense closed subsets of $P(W)$.
Let $Z=\mu(C)$. Then $Z$ has measure zero in $P(W)$ and $s^{-1}(0)\subset X$ is smooth for $[s]\in P(W)\setminus Z$.
 In general, $Z$ is not an analytic subset. 
\end{proof}

Given a globally generated holomorphic vector bundle $F$, one can always find a finite dimensional subspace $W\subset H^{0}\left(X, F\right)$ to generate all fibers of $F$. This fact is shown by the following theorem.

\begin{thm}[cf. \cite{Dembook}]
Let $F$ be a holomorphic vector bundle of rank $r$ over a complex manifold $X$.
If $F$ is globally generated, then there exists a finite dimensional subspace $W\subset H^{0}\left(X, F\right)$, $\dim W\leqslant \dim X +r$, such that W generates all fibers $F_{x}$, ~$x\in X$.
\end{thm}

\begin{rmk}
If $F$ is an ample vector bundle over a complex manifold $X$.
Suppose $F$ can be generated by sections in $W\subset H^{0}(X, F)$.
Then $W$ has no non-vanishing section if and only if $F$ has no trivial subbundles.
In fact, if $F$ has a trivial subbundle $F_{1}$, then we have an exact sequence of vector bundles
\[0\to F_{1}\to F\to Q\to 0.\]
Then $F\cong F_{1}\oplus Q$ by the vanishing theorem of Le Potier (\cite{LeP75}).
In this case $W$ has non-vanishing sections because $F$ can be generated by sections in $W$.
\end{rmk}

\begin{rmk}
Let $F$ be a globally generated homomorphic vector bundle over a complex manifold $X$.
If $\operatorname{rank}F>\dim X$, then $F$ has a trivial subbundle of rank $(\operatorname{rank}F-\dim X)$.
This is a result due to Serre. For a proof, we may refer to \cite{OSS11}.
\end{rmk}

%%%%%%%%%%%%%%%%%%%%%%%%%%%%%%%%%%%%%%%%%%%%%%%%%%%%%%%%

\section{Restriction formula on multiplier ideal sheaves}

\begin{thm}\label{thm: Res}
Let $\varphi$ be a quasi-psh function on a complex manifold $X$.
Let $F$ be a holomorphic vector bundle over $X$ and let $W$ be a finite dimensional subspace of $H^{0}\left(X, F\right)$ such that $W$ generates all fibers $F_{x}, ~x\in X$.
Let $P(W)$ denote the projective space of $W$ and 
\begin{equation*}
B=\Big\{[s]\in P(W) \big| ~S=s^{-1}(0) ~\text{is smooth and}~\mathcal{I}(\varphi|_S)=\mathcal{I}(\varphi)|_{S}\Big\}.
\end{equation*}
If $W$ has no non-vanishing sections, then $P(W)\setminus B$ has measure zero in $P(W)$.
\end{thm}

\begin{proof}
Let us consider the diagram constructed in the proof of Theorem \ref{thm: Ber}
\begin{equation}
\begin{CD}
P(E) @>>{\mu}> P(W) \\
@VV\pi V \\
X 
\end{CD}
\end{equation}
For any smooth submanifold $S\subset X$, the Ohsawa-Takegoshi extension theorem implies that 
\[\mathcal{I}(\varphi|_S)\subset\mathcal{I}(\varphi)|_{S}.\]
For the other direction, 
let $U\subset X$ be a small open subset and suppose $f\in\mathscr{O}(U)$ and
\[\int_{U}|f|^{2}e^{-2\varphi}dV<+\infty.\]
Let us consider the bundle $\pi: \pi^{-1}(U)\to U$, where $\pi^{-1}(U)\subset P(E)$.
After shrinking $U$, we may assume $\pi^{-1}(U)\cong U\times\mathbb{P}^{N-r-1}$ is a trivial bundle over $U$.
Set \[F=f\circ\pi, \quad \widetilde{\varphi}=\varphi\circ\pi.\] 
Then $F\in\mathscr{O}\left(\pi^{-1}(U)\right)$ and $\widetilde{\varphi}$ is quasi-psh on $\pi^{-1}(U)$.
Let $dV_{FS}$ be the volume form on $\mathbb{P}^{N-r-1}$ associated to the Fubini-Study metric.
Let $dV_{\pi^{-1}(U)}$ be the smooth volume form on $\pi^{-1}(U)$ induced by $dV$ and $dV_{FS}$.
Since $F$ and $\widetilde{\varphi}$ are constant along the fibers $\mathbb{P}^{N-r-1}$,
\begin{equation}
\int_{\pi^{-1}(U)}|F|^{2}e^{-2\widetilde{\varphi}}dV_{\pi^{-1}(U)}=\int_{\mathbb{P}^{N-r-1}}dV_{FS}\cdot\int_{U}|f|^{2}e^{-2\varphi}dV<+\infty
\end{equation}
by Fubini's theorem.

We first assume that $X$ is compact. Then the map $\mu: P(E)\to P(W)$ is proper.
By Theorem \ref{thm: Sard}, one can find a proper analytic subset $Z\subset P(W)$ such that
\[\mu: P(E)\setminus\mu^{-1}(Z)\longrightarrow P(W)\setminus Z\] 
is a submersion and the fibers $\mu^{-1}\left([s]\right), [s]\in P(W)\setminus Z$ are smooth.

Let $\Omega$ be a simply connected domain in $P(W)\setminus Z$.
Then $\mu: \mu^{-1}(\Omega)\to\Omega$ is a submersion and hence $\mu^{-1}(\Omega)\subset P(E)$ is diffeomorphic to the product smooth manifold $\Omega\times S_{0}$, where $S_{0}$ is a complex submanifold of $X$ determined by an element in $\Omega$.
Let $dV_{P(W)}$ be a smooth volume form on $P(W)$ and $dV_{S}$ a smooth volume form on $S_{0}$.
These two measure induce a smooth volume form $dV_{\mu^{-1}(\Omega)}$ on $\mu^{-1}(\Omega)$.
However, by shrinking $U$ and $\Omega$ smaller, the two volume forms $dV_{\pi^{-1}(U)}$ and $dV_{\mu^{-1}(\Omega)}$ are equivalent. 
So we can conclude
\begin{equation}
\int_{\pi^{-1}(U)}|F|^{2}e^{-2\widetilde{\varphi}}dV_{\mu^{-1}(\Omega)}<+\infty.
\end{equation}
By Fubini's theorem,
\begin{equation}
\begin{split}
&\int_{[s]\in\Omega}\left(\int_{U\cap\mu^{-1}\left([s]\right)}|f|^{2}e^{-2\varphi}dV_{S}\right)dV_{P(W)} \\
=&\int_{\pi^{-1}(U)\cap\mu^{-1}(\Omega)}|F|^{2}e^{-2\widetilde{\varphi}}dV_{\mu^{-1}(\Omega)} \\
\leqslant&\int_{\pi^{-1}(U)}|F|^{2}e^{-2\widetilde{\varphi}}dV_{\mu^{-1}(\Omega)}<+\infty
\end{split}
\end{equation}
and hence the set
\[N(U, \Omega, f)=\left\{[s]\in\Omega; \quad \int_{U\cap\mu^{-1}\left([s]\right)}|f|^{2}e^{-2\varphi}dV_{S}=+\infty\right\}\] 
has measure zero in $P(W)$.
For $[s]\in\Omega\setminus N(U, \Omega, f)$, we have 
\[\int_{U\cap\mu^{-1}\left([s]\right)}|f|^{2}e^{-2\varphi}dV_{S}<+\infty\] 
and hence $f|_{S}\in\mathcal{I}(\varphi|_S)$.
After shrinking $U$, we may assume that $\mathcal{I}(\varphi)|_{U}$ is globally generated by $f_{1}, \cdots, f_{m}\in\mathscr{O}(U)$ and
\[\int_{U}|f_{j}|^{2}e^{-2\varphi}dV<+\infty, \quad 1\leqslant j\leqslant m.\]
The set $N(U, \Omega)=\cup_{j=1}^{m}N(U, \Omega, f_{j})$ has measure zero. 
If $[s]\notin N(U, \Omega)$, then $f_{j}|_{S}\in\mathcal{I}(\varphi|_S)$ for all $j$ and hence $\mathcal{I}(\varphi)|_{S\cap U}=\mathcal{I}(\varphi|_{S\cap U})$. Let 
\[N(U)=\left\{[s]\in P(W)\setminus Z; ~\mathcal{I}(\varphi)|_{S\cap U}\neq\mathcal{I}(\varphi|_{S\cap U}) \right\}.\]
Then it is easy to see that the measure of $N(U)$ is zero and hence the set
\[N=\left\{[s]\in P(W)\setminus Z; ~\mathcal{I}(\varphi)|_{S}\neq\mathcal{I}(\varphi|_{S})\right\}. \]
has measure $0$.
Therefore, the set
\[P(W)\setminus B=Z\cup N \]
has measure zero in $P(W)$.

In the general case, the holomorphic map $\mu: P(E)\to P(W)$ may not be proper.
Let $C$ denote the critical set of $\mu$, which is the set of points in $X$ at which the differential $\mu_{\ast}: T_{P(E)}\to T_{P(W)}$ of $\mu$ is not surjective.
It is obvious that $C$ is an analytic subset of $P(E)$.
By Sard's theorem, the set $\mu(C)$ has measure zero in $P(W)$.
Moreover, one can show that $\mu(C)$ is a countable union of nowhere dense closed subsets of $P(W)$.
The proper mapping theorem implies $\mu(C)$ is an analytic subset of $P(W)$ in case $\mu$ is proper.
But this is no longer true if $\mu$ is not proper.

Let $p\in\pi^{-1}(U)\setminus C$ be a regular point of $\mu$.
Then the differential of $\mu$ \[\mu_{\ast}: T_{p}P(E)\to T_{\mu(p)}P(W)\] is surjective.
By the inverse function theorem, there is an open neighborhood $U_{p}\subset \pi^{-1}(U)\setminus C$ of $p$, an open neighborhood $\Omega_{p}$ of $\mu(p)$ and a domain $G_{p}\subset\mathbb{C}^{n-r}$ such that $U_{p}$ is isomorphic to the product space $\Omega_{p}\times G_{p}$ and the map $\mu|_{U_{p}}$ is given by the natural projection 
\[\Omega_{p}\times G_{p}\to\Omega_{p}.\]
In other words, $\mu^{-1}([s])\cap U_{p}$ is isomorphic to $G_{p}$.

Suppose $dV_{P(W)}$ and $dV_{G_{p}}$ are smooth volume forms on $P(W)$ and $G_{p}$ respectively.
These two measure induce a smooth volume form $dV_{U_{p}}$ on $U_{p}$.
By shrinking $U_{p}$ smaller if necessary, the two volume forms $dV_{\pi^{-1}(U)}$ and $dV_{U_{p}}$ are equivalent on $U_{p}$. 
Then we can conclude
\begin{equation}
\int_{U_{p}}|F|^{2}e^{-2\widetilde{\varphi}}dV_{U_{p}}<+\infty.
\end{equation}
By Fubini's theorem,
\begin{equation}
\int_{\Omega_{p}}\left(\int_{G_{p}}|F|^{2}e^{-2\widetilde{\varphi}}dV_{G_{p}}\right)dV_{P(W)}
=\int_{U_{p}}|F|^{2}e^{-2\widetilde{\varphi}}dV_{U_{p}}<+\infty.
\end{equation}
Thus the set
\[N(U_{p}, f)=\left\{[s]\in P(W)\setminus\mu(C); \quad \int_{U_{p}\cap\mu^{-1}([s])}|F|^{2}e^{-2\widetilde{\varphi}}dV_{G_{p}}=+\infty\right\}\] 
has measure zero in $P(W)$.

\begin{rmk}
One can define a function $\psi: \Omega_{p}\longrightarrow [-\infty, +\infty)$ by setting
\[\psi\left([s]\right)=-\log\int_{U_{p}\cap\mu^{-1}([s])}|F|^{2}e^{-2\widetilde{\varphi}}dV_{G_{p}},\]
It is easy to see that $\psi$ is upper semi-continuous and 
\[N(U_{p}, f)\cap\Omega_{p}=\psi^{-1}(-\infty).\]
So $N(U_{p}, f)$ is a pluripolar set in case $\psi$ is quasi-psh.
In fact, this is true if $F\equiv 1$ and the quasi-psh function $\varphi$ is invariant under the actions of certain Lie groups (cf. \cite{Ber98}, \cite{DZZ17}).
\end{rmk}

By the second-countability, we can choose a countable collection $\{U_{p_{j}}\}$ such that $\bigcup_{j}U_{p_{j}}=\pi^{-1}(U)\setminus C$ and the set
\[N(U_{p_{j}}, f)=\left\{[s]\in P(W)\setminus\mu(C); \quad \int_{U_{p_{j}}\cap\mu^{-1}([s])}|F|^{2}e^{-2\widetilde{\varphi}}dV_{G_{p}}=+\infty\right\}\] 
has measure zero in $P(W)$ for all $j$. 
Let \[N(f)=\bigcup N(U_{p_{j}}, f).\]
It is obvious that $N(f)$ has measure zero.

Let $K\Subset U$ be a compact subset.
If $[s]\notin\mu(C)$, then $\mu^{-1}([s])$ is smooth and the analytic subset $\mu^{-1}([s])\cap\pi^{-1}(K)$ is compact.
Since $\mu^{-1}([s])\cap C=\emptyset$, the compact set $\mu^{-1}([s])\cap\pi^{-1}(K)$ can be covered by finitely many $U_{p_{j}}$.
If, moreover, $s\notin N(f)$, then $[s]\notin N(U_{p_{j}}, f)$ and hence
\begin{equation}
\int_{U_{p_{j}}\cap\mu^{-1}([s])}|F|^{2}e^{-2\widetilde{\varphi}}dV_{G_{p}}<+\infty, \quad \forall~ j.
\end{equation}
Let $dS$ be a smooth volume form on $S=s^{-1}(0)\cong\mu^{-1}([s])$.
Then we have
\begin{equation}
\int_{S\cap K}|f|^{2}e^{-2\varphi}dV_{S}=\int_{\mu^{-1}([s])\cap\pi^{-1}(K)}|F|^{2}e^{-2\widetilde{\varphi}}dV_{S}<+\infty.
\end{equation}
In other words, $|f|^{2}e^{-2\varphi}$ is locally integrable on $S\cap U$.
Therefore, \[f|_{S}\in\mathcal{I}(\varphi|_{S\cap U}).\]
After shrinking $U$, we may assume that $\mathcal{I}(\varphi)|_{U}$ is globally generated by $f_{1}, \cdots, f_{m}\in\mathscr{O}(U)$ and
\[\int_{U}|f_{k}|^{2}e^{-2\varphi}dV<+\infty, \quad 1\leqslant k\leqslant m.\]
The set $N(U)=\cup_{k=1}^{m}N(f_{k})$ has measure zero. 
If $[s]\notin N(U)$, then $f_{k}|_{S}\in\mathcal{I}(\varphi|_S)$ for all $k$ and hence $\mathcal{I}(\varphi)|_{S\cap U}=\mathcal{I}(\varphi|_{S\cap U})$. 
Finally, we may take at most countable many $U_{i}$ covering $X$ so that $N(U_{i})\subset P(W)$ has measure zero and \[\mathcal{I}(\varphi)|_{s^{-1}(0)\cap U}=\mathcal{I}(\varphi|_{s^{-1}(0)\cap U})\] 
for $[s]\notin N(U_{k})\bigcup\mu(C)$.
Let
\[N=\left\{[s]\in P(W)\setminus\mu(C); \quad \mathcal{I}(\varphi)|_{S}\neq\mathcal{I}(\varphi|_{S})\right\}. \]
Then $N\subset\bigcup_{k}N(U_{i})$ has measure $0$.
Therefore, the set
\[P(W)\setminus B=\mu(C)\bigcup N \]
has measure zero in $P(W)$.
\end{proof}

Let $X$ be a complex manifold (not necessary compact).
A complete linear system on $X$ is defined as the set of all effective divisors linearly equivalent to some given divisor $D$. It is denoted $|D|$. 
Let $L$ be the line bundle associated to $D$. 
In the case that $X$ is compact the set $|D|$ is in natural bijection with $\left(H^{0}\left(X, L\right)\setminus\{0\}\right)/\mathbb{C}^{\ast}$ and is therefore a projective space.

A linear system $\mathfrak {d}$ is then a projective subspace of a complete linear system, so it corresponds to a vector subspace $W$ of $H^{0}\left(X, L\right)$. The dimension of the linear system $\mathfrak {d}$ is its dimension as a projective space. Hence $\dim\mathfrak{d}=\dim W-1$.

\begin{cor}
Let $\mathfrak{d}$ be a base point free linear system on a complex manifold $X$ with $1\leqslant\dim\mathfrak{d}<+\infty$, and let $\varphi$ be a quasi-plurisubharmonic function on $X$. 
If we put
\[\mathfrak{b}=\Big\{S\in\mathfrak{d}; \quad S~\text{is smooth and} ~\mathcal{I}(\varphi)|_{S}=\mathcal{I}(\varphi|_{S}) \Big\},\]
then $\mathfrak{d}\setminus\mathfrak{b}$ has measure zero in $\mathfrak{d}$.
\end{cor}

It is clear that $\mathfrak{b}$ is dense in $\mathfrak{d}$ when $\mathfrak{d}\setminus\mathfrak{b}$ has measure zero.
So the above corollary implies Theorem 1.10 in \cite{FuMa16}.

\begin{rmk}
One cannot expect that $\mathfrak{d}\setminus\mathfrak{b}$ is a countable union of analytic subsets of $\mathfrak{d}$ in general. For counterexamples, one can refer to Example 3.12 in \cite{FuMa16}.
\end{rmk}

%%%%%%%%%%%%%%%%%%%%%%%%%%%%%%%%%%%%%%%%%%%%%%%%%%%%%%%%

\section{Restriction formula on complex singularity exponents}

The Ohsawa-Takegoshi $L^{2}$ extension theorem implies the following important monotonicity result.
\begin{prop}[\cite{DeKo01}]\label{prop: mon}
Let $\varphi$ be a quasi-psh function on a complex manifold $X$, and let $Y\subset X$ be a complex submanifold such that $\varphi|_{Y}\not\equiv -\infty$ on every connected component of $Y$. Then, if $K$ is a compact subset of $Y$, we have
\begin{equation}
c_{K}\left(\varphi|_{Y}\right)\leqslant c_{K}(\varphi).
\end{equation}
\end{prop}

\begin{thm}\label{thm: cse}
Let $\varphi$ be a quasi-psh function on a complex manifold $X$.
Let $F$ be a holomorphic vector bundle over $X$ and let $W$ be a finite dimensional subspace of $H^{0}\left(X, F\right)$ such that $W$ generates all fibers $F_{x}, ~x\in X$.
Let $P(W)$ denote the projective space of $W$ and 
\begin{equation*}
Q=\Big\{[s]\in P(W) \big| ~S=s^{-1}(0) ~\text{is smooth and}~c_{x}\left(\varphi|_S\right)=c_{x}(\varphi), ~\forall x\in S\Big\}.
\end{equation*}
If $W$ has no non-vanishing sections, then $P(W)\setminus Q$ has measure zero in $P(W)$.
\end{thm}

\begin{proof}
Let $\left\{x_{j}\right\}_{j\in J}$ be a countable dense subset of $X$.
Let $g$ be a complete Riemannian metric on $X$.
Then the collection of open balls
\[B_{jk}=\left\{x\in X; ~\operatorname{dist}(x, x_{j})<\frac{1}{k}\right\}, \quad j\in J, ~k\geqslant 1\]
is an open cover of $X$ and $\overline{B}_{jk}\Subset X$ is compact.
The set
\[Z=\Big\{[s]\in P(W); \quad s^{-1}(0) ~\text{is not smooth}\Big\}\]
has Lebesgue measure zero by Theorem \ref{thm: Ber}.
Fix $c\in\mathbb{Q}\cap[0, c_{\bar{B}_{jk}}(\varphi))$ and set
\begin{equation*}
E_{j, k, c}=\left\{[s]\in P(W)\setminus Z;  ~e^{-2c\varphi|_{S}}\notin L^{1}~\text{on}~S\cap\overline{B}_{jk}\right\},
\end{equation*}
where $S=s^{-1}(0)$ is the submanifold of $X$.

We next show that $E_{j, k, c}$ has measure zero.
Again, let us consider the diagram
\begin{equation}
\begin{CD}
P(E) @>>{\mu}> P(W) \\
@VV\pi V \\
X 
\end{CD}
\end{equation}
Since $c<c_{\bar{B}_{jk}}(\varphi)$, the function $e^{-2c\varphi}$ is integrable on a neighborhood of $\overline{B}_{jk}$.
If we set $\widetilde{\varphi}=\pi^{\ast}\varphi$, then $\widetilde{\varphi}$ is quasi-psh on $P(E)$ and $e^{-2c\widetilde{\varphi}}$ is integrable on a neighborhood of $\pi^{-1}\left(\overline{B}_{jk}\right)$.
Let $C$ be the critical set of $\mu$.
As in the proof of Theorem \ref{thm: Res}, for any point $p\in\pi^{-1}\left(\overline{B}_{jk}\right)$, there is an open neighborhood $U_{p}$ of $p$ such that $U_{p}$ is isomorphic to a product space $\Omega_{p}\times G_{p}$ and the map $\mu|_{U_{p}}$ is given by the natural projection 
\[\Omega_{p}\times G_{p}\to\Omega_{p}.\]
By shrinking $U_{p}$ smaller if necessary, we may assume $e^{-2c\widetilde{\varphi}}$ is integrable on $U_{p}$.
Then we can apply Fubini's theorem to conclude the set 
\[\left\{[s]\in P(W)\setminus\mu(C); \quad e^{-2c\widetilde{\varphi}}\notin L^{1}\left(U_{p}\cap\mu^{-1}([s]) \right)\right\}\]
has measure zero.
By the second-countability, $\pi^{-1}\left(\overline{B}_{jk}\right)\setminus C$ can be covered by a countable collection $\left\{U_{p_{\ell}}\right\}$. 
So the set
\[\left\{[s]\in P(W)\setminus\mu(C); \quad e^{-2c\widetilde{\varphi}}\notin L^{1}\left(\pi^{-1}\left(\overline{B}_{jk}\right)\cap\mu^{-1}([s])\right)\right\}\]
has measure zero too.
Note that the map
\[\pi^{-1}\left(\overline{B}_{jk}\right)\cap\mu^{-1}([s])\xlongrightarrow{\pi}\overline{B}_{jk}\cap s^{-1}(0)\]
is isomorphic.
Thus we can conclude
\[N_{j, k, c}=\left\{[s]\in P(W)\setminus\mu(C); \quad e^{-2c\varphi}\notin L^{1}\left(\overline{B}_{jk}\cap s^{-1}(0)\right)\right\}\]
has measure zero.

Now we define 
\[N=\bigcup_{j, k, c}N_{j, c, k}, \quad j\in J, ~ k\geqslant 1, ~ c\in\mathbb{Q}\cap\left[0, c_{\bar{B}_{jk}}(\varphi)\right).\]
Then $N\cup Z\subset P(W)$ has measure zero.
Suppose \[[s]\in P(W)\setminus\left(Z\cup N\right), \quad S=s^{-1}(0).\]
We claim that
\[c_{x}\left(\varphi\right)=c_{x}\left(\varphi|_{S}\right), \quad \forall x\in S.\]
To see this, let $c$ be a rational number such that $0\leqslant c< c_{x}(\varphi)$. 
Then there is a neighborhood $W$ of $x$ in $X$ such that $e^{-2c\varphi}$ is integrable on $W$.
We may assume that the ball $B\left(x, \frac{2}{k}\right)$ of center $x$ and radius $\frac{2}{k}$ is contained in $W$.
Since $\{x_{j}\}_{j\in J}$ is dense in $X$, we can find an index $j\in J$ such that $\operatorname{dist}\left(x_{j}, x\right)<\frac{1}{k}$.  Then
\[x\in B_{jk}=B\left(x_{j}, \frac{1}{k}\right)\subset B\left(x, \frac{2}{k}\right)\subset W.\]
It is obvious that $e^{-2c\varphi}$ is integrable on $B_{jk}$. 
Now $S$ is smooth since $[s]\notin Z$ and $e^{-2c\varphi|_{S}}$ is integrable on $B_{jk}\cap S$ because $[s]\notin N_{j, k, c}$. It follows that $e^{-2c\varphi|_{S}}$ is integrable in a neighborhood of $x\in S$ and hence $c_{x}\left(\varphi|_{S}\right)\geqslant c$.
So we can get the desired inequlity
\[c_{x}\left(\varphi|_{S}\right)\geqslant c_{x}(\varphi).\]
Finally, we have $c_{x}\left(\varphi|_{S}\right)= c_{x}(\varphi)$ because the inequality in the opposite direction is always true by Proposition \ref{prop: mon}.
\end{proof}

We next discuss the restriction formula on jumping numbers.
Let $\varphi$ be a quasi-psh function on a complex manifold $X$ and let $\mathscr{I}\subset\mathscr{O}_{X}$ be a nonzero coherent ideal sheaf.
The jumping number $c_{x}^{\mathscr{I}}(\varphi)$ is defined as follows (see \cite{JoMu14}):
\[c_{x}^{\mathscr{I}}(\varphi)=\sup\left\{c\geqslant 0; \quad \left|\mathscr{I}\right|^{2}e^{-2c\varphi} ~\text{is locally integrable at}~ x\right\},\]
where $\left(f_{j}\right)_{1\leqslant j\leqslant N}$ are local generators of $\mathscr{I}$ and $\left|\mathscr{I}\right|^{2}=\sum_{j=1}^{N}|f_{j}|^{2}$.
When $\mathscr{I}=\mathscr{O}_{X}$, the jumping number reduces to the complex singularity exponent.
Arguments similar to those in the proof of Theorem \ref{thm: Res} and Theorem \ref{thm: cse} easily yield 

\begin{thm}\label{thm: cse}
Let $\varphi$ be a quasi-psh function on a complex manifold $X$ and let $\mathscr{I}\subset\mathscr{O}_{X}$ be a nonzero coherent ideal sheaf.
Let $F$ be a holomorphic vector bundle over $X$ and let $W$ be a finite dimensional subspace of $H^{0}\left(X, F\right)$ such that $W$ generates all fibers $F_{x}, ~x\in X$.
Let $P(W)$ denote the projective space of $W$.
If $W$ has no non-vanishing sections, then there is a measure zero set $N\subset P(W)$ such that for every $[s]\notin Z$, the subvariety $S=s^{-1}(0)$ is smooth and 
\[c_{x}^{\mathscr{I}}(\varphi)\leqslant c_{x}^{\mathscr{I}|_{S}}\left(\varphi|_S\right), \quad \forall x\in S.\]
\end{thm}

%%%%%%%%%%%%%%%%%%%%%%%%%%%%%%%%%%%%%%%%%%%%%%%%%%%%%%%%

\section{A short exact sequence}

The following theorem is a special case of Theorem 4 in \cite{Siu69}. We follow the exposition of \cite{Man82}.
\begin{thm}[Siu]
Let $X$ be a compact space and let $\mathscr{F}$ be a coherent analytic sheaf on $X$.
Then there is a locally finite family $\{Y_{i}, ~i\in I\}$ of irreducible analytic subsets of $X$ such that for each $x\in X$ the associated primes of $\mathscr{F}_{x}$ is
\[\operatorname{Ass}_{\mathscr{O}_{X, x}}{\mathscr{F}_{x}}=\{\mathfrak{p}_{x, 1}, \cdots, \mathfrak{p}_{x, r(x)}\},\]
where $\mathfrak{p}_{x, 1}, \cdots, \mathfrak{p}_{x, r(x)}$ are the prime ideals of $\mathscr{O}_{X, x}$ associated to the irreducible components of the germs $Y_{i, x}$~$i\in I$ with $x\in Y_{i}$.
\end{thm}

\begin{defn}[cf. \cite{Siu69}]
The analytic subsets $Y_{i}, ~i\in I$ of the above theorem are called analytic subsets associated to the sheaf $\mathscr{F}$. The index set $I$ is at most countable since the family $\{Y_{i}, ~i\in I\}$ is locally finite.
\end{defn}

\begin{lem}\label{lem: inj}
Let $L$ be a holomorphic line bundle over a complex manifold $X$.
Let $\mathscr{F}$ be a coherent analytic sheaf on $X$ and let $Y_{i}, ~i\in I$ be the analytic subsets associated to $\mathscr{F}$.
Suppose $s$ is a nonzero section of $L$ and $S=s^{-1}(0)$, then the sequence
\[0\longrightarrow\mathscr{F}\otimes\mathscr{O}(-S)\xrightarrow{\otimes s} \mathscr{F}\]
is exact if and only if $S\not\supset Y_{i}$ for all $i\in I$. 
\end{lem}

\begin{proof}
Locally, we may assume the section $s$ is given by a holomorphic function $f$ and the map
$\mathscr{F}\otimes\mathscr{O}(-S)\xrightarrow{\otimes s} \mathscr{F}$ is given by $\mathscr{F}\xrightarrow{f} \mathscr{F}$.

Suppose $s$ vanish on $Y_{i}$ for some $i\in I$.
Let $\mathfrak{p}_{x}$ be the prime ideal associated to an irreducible component of the germ $Y_{i, x}$.
Since $s\neq 0$, $Y_{i}$ is a proper analytic subset of $X$ and hence $\mathfrak{p}_{x}\neq 0$.
We may assume $\mathfrak{p}_{x}=\operatorname{Ann}(t_{x})$ for some nonzero $t_{x}\in\mathscr{F}_{x}$.
Since $s$ vanish on $Y_{i}$, we have $f_{x}\in\mathfrak{p}_{x}=\operatorname{Ann}(t_{x})$.
Thus $f_{x}\cdot t_{x}=0$ and hence $\mathscr{F}_{x}\xrightarrow{f_{x}} \mathscr{F}_{x}$ is not injective.

For the other direction, suppose $\mathscr{F}_{x}\xrightarrow{f_{x}} \mathscr{F}_{x}$ is not injective for some point $x\in X$.
There exists a nonzero section $\tilde{t}_{x}\in\mathscr{F}_{x}$ such that $f_{x}\cdot \tilde{t}_{x}=0$.
Then $f_{x}\in\operatorname{Ann}(\tilde{t}_{x})$.
It is easy to show that every maximal element of the family of ideals 
\[\{\operatorname{Ann}(t_{x}); ~0\neq t\in\mathscr{F}_{x}, ~\operatorname{Ann}(t_{x})\supset\operatorname{Ann}(\tilde{t}_{x})\}\] 
is an associated prime of $\mathscr{F}_{x}$ (cf. \cite{Mat80}).
Suppose $\mathfrak{p}_{x}\supset\operatorname{Ann}(\tilde{t}_{x})$ is an associated prime of $\mathscr{F}_{x}$. 
Then $f_{x}\in\mathfrak{p}_{x}$.
So the section $s$ vanishes on a component of the germ $Y_{i, x}$.
Since $Y_{i}$ is an irreducible analytic subset of $X$, we have $s^{-1}(0)\supset Y_{i}$.
\end{proof}

\begin{lem}\label{lem: gen}
Let $L$ be a holomorphic line bundle over a complex manifold $X$ and let $W$ be a finite dimensional subspace of $H^{0}\left(X, L\right)$ such that $W$ generates all fibers $L_{x}, ~x\in X$.
Let $\mathscr{F}$ be a coherent analytic sheaf on $X$ and let $Y_{i}, ~i\in I$ be the analytic subsets associated to $\mathscr{F}$.
Let $P(W)$ denote the projective space of $W$. Suppose $\dim P(W)\geqslant 1$.
Then the set
\[A=\Big\{[s]\in P(W) \big| ~s^{-1}(0)\supset Y_{i} ~\text{for some}~i\in I\Big\}\]
is a countable union of proper analytic subsets of $P(W)$. 
If, moreover, $X$ is compact, then $A$ is analytic in $P(W)$.
\end{lem}

\begin{proof}
Let $E$ be the holomorphic vector bundle given by the exact sequence
\[0\to E\to X\times W\to L\to 0.\]
Let us consider the following diagram
\begin{equation}
\begin{CD}
P(E) @>>{\mu}> P(W) \\
@VV\pi V \\
X 
\end{CD}
\end{equation}
Then we can write
\[A=\bigcup_{i\in I}\left(\bigcap_{y\in Y_{i}}\mu\left(P(E_{y})\right)\right).\]
By the definition of $\mu$, the set $\mu\left(P(E_{y})\right)=P(E_{y})\subset P(W)$ is a hyperplane and hence $\bigcap_{y\in Y_{i}}\mu\left(P(E_{y})\right)$ is a linear subspace of $P(W)$.
Since the family $\{Y_{i}, ~i\in I\}$ is locally finite, the index set $I$ is at most countable. 
So $A$ is at most a countable union of proper analytic subsets of $P(W)$.
If $X$ is compact, then $I$ is finite and hence $A$ is an analytic subset of $P(W)$.
\end{proof}

\begin{thm}
Let $L$ be a holomorphic line bundle over a complex manifold $X$ 
and let $W$ be a finite dimensional subspace of $H^{0}\left(X, L\right)$ such that $W$ generates all fibers $L_{x}, ~x\in X$. 
We denote by $P(W)$ the projective space of $W$.
Let $\varphi$ be a quasi-psh function on $X$ and let $Y_{i}, ~i\in I$ be the analytic subsets associated to $\mathscr{O}_{X}/\mathcal{I}(\varphi)$.
Suppose $L$ is not a trivial line bundle. Then:
\begin{itemize}
\item
The set
\[A=\Big\{[s]\in P(W) \big| ~s^{-1}(0)\supset Y_{i} ~\text{for some}~i\in I\Big\}\]
is a countable union of proper analytic subsets of $P(W)$.
If, moreover, $X$ is compact, then $A$ is analytic in $P(W)$.
\item
Suppose $[s]\in P(W)$ and $S=s^{-1}(0)$, then the sequence
\begin{equation}
 0\longrightarrow\mathcal{I}(\varphi)\otimes\mathscr{O}(-S)\longrightarrow\mathcal{I}(\varphi)\longrightarrow\mathcal{I}(\varphi)|_{S}\longrightarrow 0
\end{equation}
is exact if and only if $[s]\notin A$.
\end{itemize}
\end{thm}

\begin{proof}
By Lemma \ref{lem: gen}, we only need to prove the second statement.
Let $\mathscr{J}$ be the ideal sheaf of $\mathscr{O}_{X}$ defined by $s\in W$.
There is a natural exact sequence
\begin{equation}
 0\longrightarrow\mathcal{I}(\varphi)\otimes\mathscr{O}(-S)\longrightarrow\mathcal{I}(\varphi)\longrightarrow\mathcal{I}(\varphi)\otimes\left(\mathscr{O}_{X}/\mathscr{J}\right)\longrightarrow 0.
\end{equation}
By definition, the ideal sheaf $\mathcal{I}(\varphi)|_{S}$ is the image of the map
\[\rho: \mathcal{I}(\varphi)\otimes\left(\mathscr{O}_{X}/\mathscr{J}\right)\to\mathscr{O}_{X}\otimes\left(\mathscr{O}_{X}/\mathscr{J}\right)\xrightarrow{\cong}\mathscr{O}_{X}/\mathscr{J}.\]
To obtain the desired exact sequence, we only need to consider the injectivity of the map $\rho$.
From the short exact sequence
 \[0\longrightarrow\mathcal{I}(\varphi)\longrightarrow\mathscr{O}_{X}\longrightarrow\mathscr{O}_{X}/\mathcal{I}(\varphi)\longrightarrow 0\]
we can obtain a long exact sequence
\[0\to\operatorname{Tor}\left(\mathscr{O}_{X}/\mathcal{I}(\varphi), \mathscr{O}_{X}/\mathscr{J}\right)\to\mathcal{I}(\varphi)\otimes\left(\mathscr{O}_{X}/\mathscr{J}\right)\to\mathscr{O}_{X}/\mathscr{J}.\]
Thus $\rho$ is injective if and only if 
\[\operatorname{Tor}\left(\mathscr{O}_{X}/\mathcal{I}(\varphi), \mathscr{O}_{X}/\mathscr{J}\right)=0.\]
From the short exact sequence
 \[0\longrightarrow\mathscr{J}\longrightarrow\mathscr{O}_{X}\longrightarrow\mathscr{O}_{X}/\mathscr{J}\longrightarrow 0\]
we have another exact sequence
\[0\to\operatorname{Tor}\left(\mathscr{O}_{X}/\mathcal{I}(\varphi), \mathscr{O}_{X}/\mathscr{J}\right)\to\left(\mathscr{O}_{X}/\mathcal{I}(\varphi)\right)\otimes\mathscr{J}\to\mathscr{O}_{X}/\mathcal{I}(\varphi).\]
Therefore, the injectivity of $\rho$ is equivalent to the injectivity of the map
\[\left(\mathscr{O}_{X}/\mathcal{I}(\varphi)\right)\otimes\mathscr{J}\to\mathscr{O}_{X}/\mathcal{I}(\varphi).\]
Since the ideal sheaf $\mathscr{J}=\mathscr{O}(-S)$, the above map can be written as
\[\left(\mathscr{O}_{X}/\mathcal{I}(\varphi)\right)\otimes\mathscr{O}(-S)\xrightarrow{\otimes s}\mathscr{O}_{X}/\mathcal{I}(\varphi).\]
By Lemma \ref{lem: inj}, we can conclude that $\rho$ is injective if and only $[s]\notin A$.
Therefore, the sequence 
\begin{equation*}
 0\longrightarrow\mathcal{I}(\varphi)\otimes\mathscr{O}(-S)\longrightarrow\mathcal{I}(\varphi)\longrightarrow\mathcal{I}(\varphi)|_{S}\longrightarrow 0
\end{equation*}
is exact if and only if $[s]\notin A$.
\end{proof}

\begin{cor}
Let $L$ be a holomorphic line bundle over a complex manifold $X$. 
Let $W$ be a finite dimensional subspace of $H^{0}\left(X, L\right)$ such that $W$ generates all fibers $L_{x}, ~x\in X$. 
We denote by $P(W)$ the projective space of $W$.
Let $\varphi$ be a quasi-psh function on $X$.
Suppose $L$ is not a trivial line bundle.
Then there exists a measure zero set $N\subset P(W)$ such that for each $[s]\in P(W)\setminus Z$ the analytic subset $S=s^{-1}(0)$ is smooth, $\mathcal{I}(\varphi|_S)=\mathcal{I}(\varphi)|_{S}$ and the sequence
\begin{equation}
 0\longrightarrow\mathcal{I}(\varphi)\otimes\mathscr{O}(-S)\longrightarrow\mathcal{I}(\varphi)\longrightarrow\mathcal{I}(\varphi|_{S})\longrightarrow 0
\end{equation}
is exact.
\end{cor}

The following result is essentially due to Cao (\cite{Cao14}) and Guan-Zhou (\cite{Guan-Zhou15a}).
\begin{thm}
Let $L$ be a holomorphic line bundle over a complex manifold $X$. 
Let $W$ be a finite dimensional subspace of $H^{0}\left(X, L\right)$ such that $W$ generates all fibers $L_{x}, ~x\in X$. 
Then there exists a measure zero set $N\subset P(W)$ such that for each $[s]\in P(W)\setminus Z$ the following statements hold 
\begin{enumerate}[a)]
\item
 the multiplier ideal sheaf $\mathcal{I}(\varphi)$ can be written as
\begin{equation}\label{eq:adj}
\mathcal{I}(\varphi)_{x}=\left\{f\in\mathscr{O}_{X, x}; ~\exists ~U_{x} ~\text{such that} ~\int_{U_{x}}\frac{|f|^{2}}{|s|^{2(1-\varepsilon)}}e^{-2(1+\sigma)\varphi}\mathrm{d}V<+\infty\right\},
\end{equation}
for $0<\sigma\leqslant\sigma_{0}$ and $0<\varepsilon<\sigma$, where $dV$ is a smooth volume form on $X$,
\item
the divisor $S=s^{-1}(0)$ is smooth,
\item
the following sequence
\begin{equation*}
 0\longrightarrow\mathcal{I}(\varphi)\otimes\mathcal{O}(-S)\longrightarrow\mathcal{I}(\varphi)\longrightarrow\mathcal{I}(\varphi|_{S})\longrightarrow 0
\end{equation*}
is exact.
\end{enumerate}
\end{thm}

\begin{proof} 
We can choose $s_{1},s_{2},\cdots,s_{N}$ to be a basis of $W$ so that $\sum_{j=1}^{N}|s_{j}(x)|^{2}\neq0$ for any $x\in X$.
Let $(\tau_{1},\tau_{2},\cdots,\tau_{N})$ be the coordinate of $\mathbb{C}^{N}$.
Suppose \[f\in\mathcal{I}(\varphi)_{x}=\mathcal{I}\left((1+\sigma)\varphi\right)_{x}, \quad 0<\sigma<\sigma_{0}.\] 
Then
\begin{equation*}
\begin{split}
&\int_{\sum_{j=1}^{N}|\tau_{j}|^{2}=1}\mathrm{d}\tau\int_{U_{x}}\frac{|f|^{2}}{|\sum_{j=1}^{N}\tau_{j}s_{j}|^{2(1-\varepsilon)}}e^{-2(1+\sigma)\varphi}\mathrm{d}V \\
=& \int_{U_{x}}\frac{|f|^{2}}{\left|\sum_{j=1}^{N}|s_{j}(y)|^{2}\right|^{1-\varepsilon}}e^{-2(1+\sigma)\varphi}\mathrm{d}V\int_{\sum_{j=1}^{N}|\tau_{j}|^{2}=1}\frac{\mathrm{d}\tau}{\left|\sum_{j=1}^{N}\tau_{j}\frac{s_{j}(y)}{\sqrt{\sum_{j=1}^{N}|s_{j}(y)|^{2}}}\right|^{2(1-\varepsilon)}} \\
=& \int_{U_{x}}\frac{|f|^{2}}{\left|\sum_{j=1}^{N}|s_{j}(y)|^{2}\right|^{1-\varepsilon}}e^{-2(1+\sigma)\varphi}\mathrm{d}V\int_{\sum_{j=1}^{N}|\tau_{j}|^{2}=1}\frac{\mathrm{d}\tau}{|\tau_{1}|^{2(1-\varepsilon)}}<+\infty.
\end{split}
\end{equation*}
For the last equality, one can change coordinate via unitary transformation so that
\begin{equation*}
\int_{\sum_{j=1}^{N}|\tau_{j}|^{2}=1}\frac{\mathrm{d}\tau}{\left|\sum_{j=1}^{N}\tau_{j}\frac{s_{j}(y)}{\sqrt{\sum_{j=1}^{N}|s_{j}(y)|^{2}}}\right|^{2(1-\varepsilon)}}=\int_{\sum_{j=1}^{N}|\tau_{j}|^{2}=1}\frac{\mathrm{d}\tau}{|\tau_{1}|^{2(1-\varepsilon)}}<+\infty.
\end{equation*}
By Fubini's theorem, we can choose $(\tau_{1},\tau_{2},\cdots,\tau_{N})\in\mathbb{C}^{N}$ outside a measure zero set such that the section $s=\sum_{j=1}^{N}\tau_{j}s_{j}$ satisfy
\begin{equation*}
\int_{U_{x}}\frac{|f|^{2}}{|\sum_{j=1}^{N}\tau_{j}s_{j}|^{2(1-\varepsilon)}}e^{-2(1+\sigma)\varphi}\mathrm{d}V<+\infty.
\end{equation*}
Then we can choose $[\tau_{1}: \tau_{2}: \cdots :\tau_{N}]\in\mathbb{C}P^{N-1}$ outside a measure set in $P(W)$.
The first statement is proved. 

By Theorem \ref{thm: Ber} and Theorem \ref{thm: Res}, we can choose $S$ outside a measure zero set so that $S$ is smooth and the restriction $\mathcal{I}(\varphi)\to\mathcal{I}(\varphi|_{S})$ is well-defined.

If $f\in\mathcal{I}(\varphi|_{S})_{x}$, then the Ohsawa-Takegoshi extension theorem implies that 
there exists a $F\in\mathcal{I}(\varphi)_{x}$ such that $F|_{S}=f$ near $x$.
So the the restriction $\mathcal{I}(\varphi)\to\mathcal{I}(\varphi|_{S})$ is surjective.
For the exactness of the middle term, let $f\in\mathcal{I}(\varphi)_{x}$ whose restriction in $\mathcal{I}(\varphi|_{S})_{x}$ is zero.
Then $f$ vanish along $S$ near $x$, hence $\frac{f}{s}\in\mathscr{O}_{x}$. So we can conclude that
\begin{equation}
\int_{U_{x}}\frac{|f|^{2}}{|s|^{4-\eta}}\mathrm{d}V=\int_{U_{x}}\left|\frac{f}{s}\right|^{2}\cdot\frac{1}{|s|^{2-\eta}}\mathrm{d}V<+\infty, \quad \eta>0.
\end{equation}
By H\"{o}lder's inequality, we have
\begin{equation*}
\int_{U_{x}}\frac{|f|^{2}}{|s|^{2}}e^{-2(1+\delta)\varphi}\mathrm{d}V
\leqslant\left(\int_{U_{x}}\frac{|f|^{2}}{|s|^{2(1-\varepsilon)}}e^{-2(1+\sigma)\varphi}\mathrm{d}V\right)^{\frac{1+\delta}{1+\sigma}}\left(\frac{|f|^{2}}{|s|^{\alpha}}\mathrm{d}V\right)^{\frac{\sigma-\delta}{1+\sigma}},
\end{equation*}
where
\begin{equation*}
\alpha=\left[2-2(1-\varepsilon)\frac{1+\delta}{1+\sigma}\right]\frac{1+\sigma}{\sigma-\delta}.
\end{equation*}
If we choose $0<\delta<\frac{\sigma-\varepsilon}{1+\varepsilon}<\sigma$, then $\alpha<4$ and hence
\begin{equation*}
\int_{U_{x}}\frac{|f|^{2}}{|s|^{2}}e^{-2(1+\delta)\varphi}\mathrm{d}V<+\infty.
\end{equation*}
This shows that $\frac{f}{s}\in\mathcal{I}\left((1+\delta)\varphi\right)_{x}=\mathcal{I}(\varphi)_{x}$.
Thus $f\in\mathcal{I}(\varphi)_{x}\otimes\mathscr{O}(-S)_{x}$.
\end{proof}

%%%%%%%%%%%%%%%%%%%%%%%%%%%%%%%%%%%%%%%%%%%%%%%%%%%%%%%%%%%%%%%%%%%%


\begin{thebibliography}{99}


\bibitem{Ber98} B. Berndtsson, 
{\em Prekopa’s theorem and Kiselman’s minimum principle for plurisubharmonic functions}, 
Math. Ann. 312 (1998), no. 4, 785–792.


\bibitem{Cao14} J. Y. Cao,
{\em Numerical dimension and a Kawamata-Viehweg-Nadel type vanishing theorem on compact K\"{a}hler manifolds},
Compositio Math. 150 (2014), 1869-1902.



\bibitem{Dem12} J. P. Demailly,
{\em Analytic Methods in Algebraic Geometry},
Surveys of Modern Mathematics 1, International Press, Somerville, MA; Higher Education Press, Beijing, 2012.


\bibitem{Dembook} J.-P. Demailly, 
{\em Complex analytic and differential geometry}, 
electronically accessible at 
\url{http://www-fourier.ujf-grenoble.fr/~demailly/books.html}.


\bibitem{DEL00} J.-P. Demailly, L. Ein and R. Lazarsfeld, 
{\em A subadditivity property of multiplier ideals}, 
Michigan Math. J. 48 (2000), 137-156.


\bibitem{DeKo01} J.-P. Demailly, J. Koll\'ar, 
{\em Semi-continuity of complex singularity exponents and K\"ahler–Einstein metrics on Fano orbifolds}, 
Ann. Sci. \'Ec. Norm. Sup\'er. (4) 34 (2001), no. 4, 525-556.


\bibitem{DZZ17} F. Deng, H. Zhang, X. Zhou, 
{\em Positivity of character subbundles and minimum principle for noncompact group actions.} 
Math. Z., 286(1-2):431–442, 2017.



\bibitem{Fuj17} O. Fujino,
{\em Relative Bertini type theorem for multiplier ideal sheaves},
arXiv:1709.01406




\bibitem{FuMa16} O. Fujino, S.-I. Matsumura, 
{\em Injectivity theorem for pseudo-effective line bundles and its applications},
arXiv:1605.02284v2.



\bibitem{GPR94}
H. Grauert, Th. Peternell, R. Remmert, editors. 
{\em Several complex variables}, 
VII. volume 74 of Encyclopaedia of Mathematical Sciences. Springer-Verlag, Berlin, 1994. Sheaf-theoretical methods in complex analysis.


\bibitem{Guan-Zhou15a} Q. A. Guan, X. Y. Zhou,
{\em A proof of Demailly's strong openness conjecture},
Ann. of Math, 182 (2015), 605-616.

\bibitem{Guan-Zhou15b} Q. A. Guan, X. Y. Zhou,
{\em Effectiveness of Demailly's strong openness conjecture and related problems},
Invent. Math,  202 (2015), no. 2, 635-676.


\bibitem{Guan-Zhou20} Q. A. Guan, X. Y. Zhou,
{\em Restriction formula and subadditivity property related to multiplier ideal sheaves},
J. reine angew. Math., 769, (2020), 1-33.


\bibitem{JoMu14} M. Jonsson, M. Musta\c{t}\u{a},
{\em An algebraic approach to the openness conjecture of Demailly and Koll\'ar},
J. Inst. Math. Jussieu 13 (2014), no. 1, 119-144.


\bibitem{LeP75} J. Le Potier, 
{\em Annulation de la cohomolgie \`a valeurs dans un fibr\'e vectoriel holomorphe positif de rang quelconque}, 
Math. Ann. 218 (1975), 35-53.


\bibitem{Man82}M. Manaresi, 
{\em Sard and Bertini type theorems for complex spaces}, 
Annali di Matematica pura ed applicata 131, (1982), 265-279.


\bibitem{Mat80} H. Matsumura,
{\em Commutative algebra},
Second edition, Mathematics Lecture Note Series, 56, Benjamin/Cummings Publishing Co., Inc., Reading, Mass., 1980.


\bibitem{Nadel} A. M. Nadel,
{\em Multiplier ideal sheaves and existence of K\"{a}hler-Einstein metrics of positive scalar curvature},
Ann. of math., 132 (1990), 613-625.


\bibitem{OSS11} C. Okonek, M. Schneider, H. Spindler, 
{\em Vector bundles on complex projective spaces (with Appendix by S. I. Gelfand)}, 
Birkh\"auser/Springer 2011.


\bibitem{Rem57} R. Remmert,
{\em Holomorphe und meromorphe Abbildungen komplexer R\"aume},
Math. Ann., 133 (1957), 328-370.

\bibitem{Siu69} Y. T. Siu, 
{\em Noether-Lasker Decomposition of Coherent Analytic Subsheaves}, 
Trans. of the Amer. Math. Soc., 135, (1969), 375-385.


\end{thebibliography}
\end{document}